\documentclass[12pt,reqno]{amsart}

\usepackage[margin=1in]{geometry}
\usepackage{enumerate}
\usepackage{ifpdf}
\usepackage{amsmath}
\usepackage{amsfonts}
\usepackage{amssymb}
\usepackage{amsthm}
\usepackage{amscd}
\usepackage[hyperfootnotes=false]{hyperref}
\usepackage{setspace}
\usepackage{amsrefs}
\usepackage{graphicx}

\usepackage{tikz}
\usepackage{rotating}

\newcommand{\ignore}[1]{}

\renewcommand{\Re}{\operatorname{Re}}

\newcommand{\abs}[1]{\left\lvert {#1} \right\rvert}
\newcommand{\norm}[1]{\left\lVert {#1} \right\rVert}

\newcommand{\C}{{\mathbb{C}}}
\newcommand{\R}{{\mathbb{R}}}

\newcommand{\N}{{\mathbb{N}}}

\newcommand{\bC}{{\mathbb{C}}}

\newcommand{\bN}{{\mathbb{N}}}

\newcommand{\bP}{{\mathbb{P}}}

\newcommand{\sM}{{\mathcal{M}}}

\newcommand{\sL}{{\mathcal{L}}}

\newcommand{\sV}{{\mathcal{V}}}

\newcommand{\sZ}{{\mathcal{Z}}}

\newcommand{\rank}{\operatorname{rank}}

\newtheorem{thm}{Theorem}[section]
\newtheorem*{thmnonum}{Theorem}

\newtheorem{prop}[thm]{Proposition}

\newtheorem{cor}[thm]{Corollary}

\newtheorem{lemma}[thm]{Lemma}

\theoremstyle{definition}
\newtheorem{defn}[thm]{Definition}
\newtheorem{example}[thm]{Example}

\theoremstyle{remark}
\newtheorem{remark}[thm]{Remark}

\author{Dusty Grundmeier}
\thanks{The first author acknowledges support from NSF grants DMS 0838434 and NSF RTG-0602191.}
\address{Department of Mathematics, University of Michigan,
Ann Arbor, MI 48109, USA}
\email{grundmer@umich.edu}

\author{Ji\v{r}\'i Lebl}
\thanks{The second author was in part supported by NSF grant DMS 0900885.}
\address{Department of Mathematics, University of Wisconsin, 
Madison, WI 53706, USA}
\email{lebl@math.okstate.edu}

\author{Liz Vivas}
\address{Department of Mathematics, Purdue University, 
West Lafayette, IN 47907, USA}
\email{lvivas@math.purdue.edu}

\dedicatory{Dedicated to the memory of Professor M.~Salah Baouendi}

\date{October 9, 2013}

\ifpdf
\hypersetup{
pdftitle={Bounding the rank of Hermitian forms and rigidity for CR mappings of hyperquadrics},
pdfauthor={Dusty Grundmeier, Jiri Lebl, and Liz Vivas}
}
\fi

\title[Rank of Hermitian forms and rigidity of CR maps]
{Bounding the rank of Hermitian forms and rigidity for CR mappings of hyperquadrics}

\begin{document}


\begin{abstract}
Using Green's hyperplane restriction theorem,
we prove that the rank of a Hermitian form on the space of holomorphic
polynomials is bounded by a constant depending only on the maximum rank
of the form restricted to affine manifolds.
As an application we prove a rigidity theorem
for CR mappings between hyperquadrics in the spirit of
the results of Baouendi-Huang and Baouendi-Ebenfelt-Huang.
Given a real-analytic CR mapping of a hyperquadric (not equivalent to a sphere)
to another hyperquadric $Q(A,B)$, either the image of the mapping
is contained in a complex affine subspace, or $A$ is bounded by
a constant depending only on $B$.
Finally, we prove a stability result about existence of
nontrivial CR mappings of hyperquadrics.  That is, as long as both $A$
and $B$ are sufficiently large and comparable, then there exist CR mappings
whose image is not contained in a hyperplane.
The rigidity result also extends when mapping to
hyperquadrics in infinite dimensional Hilbert-space.
\end{abstract}

\maketitle



\section{Introduction} \label{section:intro}

A real-valued polynomial, or a real-analytic function,
$r(z,\bar{z})$ on $\C^n$
can be regarded as a Hermitian form by considering the
matrix of coefficients of the series.  That is, in multi-index
notation, write
\begin{equation} \label{eq:series}
r(z,\bar{z}) = \sum_{\alpha \beta} c_{\alpha \beta} z^\alpha \bar{z}^\beta .
\end{equation}
The matrix $C = {[ c_{\alpha \beta} ]}_{\alpha \beta}$ is Hermitian symmetric
if and only if $r$ is real-valued ($C$ is uniquely determined by $r$).

When we apply linear algebra terminology
(such as \emph{rank}, \emph{eigenvalues}, \emph{signature},
or \emph{positive semidefinite}) to $r$, we simply refer to
the underlying matrix $C$.
When $r$ is a polynomial, the terminology is
obvious.  When $r$ is real-analytic, then $C$ is an
infinite matrix.  After possibly rescaling so that the series converges 
in a neighborhood of the closed unit polydisc, $C$ defines a 
Hermitian trace-class operator, and the terminology easily extends to the real-analytic
case; see section \ref{section:racase}.
While the matrix depends on the point where we expand the series,
we show that the
rank does not change under a biholomorphic change of coordinates.
In particular the rank does not depend on the point
where we expand the series.

In order to state the first result, we define the rank of the
restriction to an affine manifold.  The Grassmannian $G(m,n)$ is the manifold of all $m$-dimensional
linear subspaces of $\bC^n$, and the affine Grassmannian $G_{m,n}$ is the manifold of all affine $m$-dimensional subspaces
in $\bC^n$.  Notice that $G_{m,n}$ is an open set in $G(m+1,n+1)$.

Let $U \subset \C^n$ be a convex neighborhood of the origin, and define
$U^* := \{ z : \bar{z} \in U \}$.  Suppose 
that the series \eqref{eq:series} converges in $U \times U^*$.
Let $L$ be an affine $m$-plane in $\bC^n$ (i.e. an element of $G_{m,n}$).  Let
$E \colon \C^m \to \C^n$ be an affine embedding of $L$ in $\bC^n$. If $L$ intersects $U$, then
define
\begin{equation}
\rank r|_L := \rank r \circ E .
\end{equation}
It is not hard to show that the rank does not depend on the particular
embedding $E$, and therefore the notation is well-defined.  Furthermore
as $U$ is convex, $U \cap L$ is always connected.
If $L$ does not intersect $U$, then define 
\begin{equation}
\rank r|_L := -\infty.
\end{equation}
Let $\N_0 = \N \cup \{ 0 \}$.

\begin{thm} \label{thm:rankres}
Let $n\geq 2$ and let $1 \leq m \leq n-1$.  Let $r(z,\bar{z})$ be a real-analytic function defined
in a convex neighborhood $U \subset \C^n$ of the origin such that the complexified
power series converges in $U \times U^*$.  If 
\begin{equation}
\sup_{L \in G_{m,n}} \,\rank r|_L < \infty,
\end{equation}
then $\rank r < \infty$.

Moreover, there exists a function $R_{m,n} \colon \N_0 \to \N_0$ such that for any such
real-analytic $r$,
\begin{equation}
\rank r \leq R_{m,n}\left(
\max_{L \in G_{m,n}} \,\rank r |_L
\right) .
\end{equation}
%
%
\end{thm}





It is worthwhile to state the theorem for bihomogeneous polynomials.
A polynomial $r(z,\bar{z})$ is said to be \emph{bihomogeneous}
of \emph{bidegree} $(d,d)$
if $r(tz,\bar{z}) = r(z,t\bar{z}) = t^d r(z,\bar{z})$.
When $r$ is a
polynomial then the coefficient matrix is simply a finite matrix.  The
function $R_{m,n}$ in the following theorem is the
same as above.

\begin{thm} \label{finiterankres}  
Let $n \geq 2$, $z\in \C^{n+1}$, and $1 \leq m \leq n-1$.  Then there exists a function
$R_{m,n} \colon \N_0 \to \N_0$ such that for any bihomogeneous polynomial $r(z,\bar{z})$, 
\begin{equation}
\rank r \leq R_{m,n}\left(
\max_{L \in G(m+1,n+1)} \,\rank r|_L
\right) .
\end{equation}
\end{thm}



As an application of Theorem \ref{thm:rankres}, we prove a rigidity result for mappings of 
hyperquadrics.
The hyperquadric $Q(a,b) \subset \C^{a+b}$ is the set defined by
\begin{equation}
Q(a,b) := \Bigl\{ z \in \C^{a+b} : \sum_{j=1}^a \abs{z_j}^2 - 
\sum_{j=a+1}^{a+b} \abs{z_j}^2 = 1 \Bigr\} .
\end{equation}
When $b=0$, then $Q(a,0)$ is simply the sphere.
Hyperquadrics and spheres are the CR analogues of
flat euclidean space from Riemannian geometry.
In CR geometry, however, there is no analogue of the Nash embedding theorem.
Therefore, a natural question in CR
geometry is to study the CR mappings
$f \colon U \to Q(A,B)$ for a CR manifold $U$.  In this paper
we take an open subset $U \subset Q(a,b)$.
We study hyperquadrics $Q(a,b)$ that are not equivalent to the sphere.
By a theorem of Lewy~\cite{Lewy:ext}
a CR function on $Q(a,b)$ extends to a holomorphic function of both sides of $Q(a,b)$.
It is therefore enough to consider real-analytic CR
mappings, or in other words restrictions of holomorphic mappings.
Therefore we study holomorphic
mappings of a neighborhood of $U$ in $\C^{a+b}$ to $\C^{A+B}$ that
take $U \subset Q(a,b)$ to $Q(A,B)$.  See the books \cites{BER:book, DAngelo:CR} for
more background information.

When considering mappings between hyperquadrics, we need to consider
problems of positivity rather than just rank.  See
\cites{DV, D:hilbertadvances} for more on positivity conditions.

After a linear fractional
change of coordinates we can
assume that $a > b$ and $A > B$.  It is possible that the change of
coordinates has a pole on $U$, but as $Q(a,b)$ is a real hypersurface,
there still has to exist a perhaps smaller dense open set $V \subset U$ such that 
$f \colon V \to Q(A,B)$ is a CR mapping.

The study of CR mappings between spheres has a long history.
When the target is also a sphere ($B=0$)
and the codimension $A-a$ is small, then certain strong rigidity results can be
proven; see for example \cite{HJX} and the references within.  Forstneri\v{c}
proved that sufficiently smooth CR mappings of spheres must be rational of
degree bounded by a constant depending only on the dimensions involved
\cite{Forstneric}.

When the target dimension is large, there is less rigidity.
When both the source and target are spheres, increasing the dimension of the target
always adds new CR mappings as long as the dimension is large enough;
see \cite{DL:complex}.  Furthermore, if the
source is a sphere and the target is a hyperquadric not equivalent to a
sphere, then for large target dimension not only do we always get new
mappings, but we get new rational mappings of arbitrarily large degree;
see \cite{DL:hermsym}.

We therefore concentrate on the case when $b,B \geq 1$.
Baouendi and Huang \cite{BH} proved
that if $b=B$, then $f$ must be a linear
embedding.
Later, Baouendi, Ebenfelt, and Huang \cite{BEH:hq} proved that 
if $B < 2b$, then after a change of coordinates
$f$ can be written as
\begin{equation} \label{beh:eq}
(z_1,\ldots,z_n) \mapsto
(z_1,\ldots,z_a,\psi(z),0,z_{a+1},\ldots,z_{a+b},\psi(z),0) ,
\end{equation}
where an arbitrary CR mapping $\psi$ and $0$ are vector-valued functions with the
right number of components.
In particular, unless $a=A$ and $b=B$, $f(U)$ is contained in a complex
hyperplane, where by a \emph{complex hyperplane} we mean a complex affine
manifold of complex codimension one.
The mapping \eqref{beh:eq} can be written, after an affine change of coordinates, as
the identity mapping direct sum an arbitrary
CR mapping going into a lower dimensional ambient space.

For any mapping $f$ of hyperquadrics,
if $f(U)$ lies in a complex hyperplane, then we apply an affine change of coordinates
and obtain a mapping $\tilde{f}$ from $U$ to $Q(A',B') \times \C^k$ for some
$A' \leq A$
and $B' \leq B$.  We can study the first $A'+B'$ components of $\tilde{f}$
as a mapping from $Q(a,b)$ to $Q(A',B')$,
as the last $k$ components of $\tilde{f}$ are arbitrary.  Therefore, it is
natural to study those mappings where $f(U)$ is not contained in a complex
hyperplane.

\begin{thm} \label{thm:rigidity}
Let $a >b \geq 1$, $U \subset Q(a,b)$ be a connected open set, and 
$f \colon U \to Q(A,B)$ be a real-analytic CR mapping such that
$f(U)$ does not lie in a complex hyperplane, then
\begin{equation}
A \leq N(a,b,B) ,
\end{equation}
where $N = N(a,b,B)$ is a constant depending only on $a$, $b$, and $B$.
\end{thm}

Note that the hypotheses on $Q(a,b)$ mean that $Q(a,b)$ is
not equivalent to a sphere.  If $Q(a,b)$ were
equivalent to a sphere, there is no rigidity.  In this sense, the
theorem is optimal.

We are interested in what happens when $B \geq 2b$.  In this
situation, there do exist nontrivial mappings that do not map to 
a hyperplane.  See sections
\ref{section:construct} and \ref{section:stability} for a general
method of constructing mappings.

The bound that we obtain on $A$ is not sharp.  For example when $a=2$ and
$b=1$, it is possible to use the method in this paper to obtain an explicit
bound
\begin{equation}
A \leq K_{1,3}(B+1)=K_3(K_2(B+1)) .
\end{equation}
See section \ref{section:greens} for the definition of $K_n(k)$.
When $B=1$ we obtain $A \leq 4$,
though we know that $A = 2$ by the result of Baouendi-Huang.  We dispense
with finding explicit formulas for $N(a,b,B)$; however, we briefly describe
the asymptotics.  In section \ref{section:racase} we show that
$K_n(k) \leq \widetilde{K}_n(k)=
\frac{n-1}{n} \, k^{\frac{n}{n-1}} + \frac{1}{n} \,k$,
and furthermore, in section \ref{section:crmaps}, we
show
\begin{equation}
\begin{split}
A\leq N(a,b,B) & \leq  K_{a+b} \circ K_{a+b-1} \circ \cdots \circ
K_{b+1}(B+1) \\
& \leq 
\widetilde{K}_{a+b} \circ \widetilde{K}_{a+b-1} \circ \cdots \circ
\widetilde{K}_{b+1}(B+1) \\
& =
\left(
\prod_{\ell=0}^{a-1}
\left(\frac{a+b-1-\ell}{a+b-\ell}\right)^%
{\frac{a+b}{a+b-\ell}}
\right)
{B}^{\frac{a+b}{b}}
+
\text{(lower order terms in $B$)} .
\end{split}
\end{equation}
Thus the bound $N(a,b,B)$ is itself bounded by a polynomial in $B$ of degree
at most $\frac{a+b}{b}$.

Unlike in the sphere to sphere case, there do exist non-rational mappings,
and rational mappings of arbitrarily high degree.
In
section
\ref{section:stability}, we 
prove a stability result about existence of hyperquadric mappings 
that do not map into hyperplanes as long as both $A$ and $B$ are sufficiently
large and comparable.  The proof is constructive and all constructed mappings are
monomial; that is, all components are single monomials.

The stability result discussed above shows that Theorem~\ref{thm:rigidity}
is optimal in the sense that super-rigidity as studied by
Baouendi-Huang and 
Baouendi-Ebenfelt-Huang only appears in small codimension.  In general
the best constant $N$ must grow asymptotically at least as a linear function
of $B$.  The methods explored in this paper do not readily give the
best bound for $N$ as we have mentioned above, but we do obtain that
$N$ asymptotically grows no faster than a polynomial in $B$.

Finally, we also note that the rigidity result extends when the target
is an infinite dimensional hyperquadric $Q(\infty,B)$, where 
$B \in \N_0 \cup \{ \infty \}$; see section~\ref{section:hilbert}.  Lempert~\cite{Lempert}
shows that any strictly pseudoconvex real-analytic compact hypersurface can be
mapped into a sphere in a possibly infinite dimensional Hilbert space.
We show that if $Q(a,b)$ is not equivalent to a sphere, then for a CR
mapping $Q(a,b) \to Q(\infty,B)$ where the image is not contained in a
hyperplane we must have $B = \infty$.  However, we
also construct mappings whose image is not contained in a hyperplane
from a real hypersurface $M \subset \C^n$
with an indefinite nondegenerate Levi-form to
$Q(\infty,B)$ for finite $B$.

The authors would like to acknowledge John D'Angelo for many useful
conversations on the subject and suggestions related to this project.
The second author would also like to thank Peter Ebenfelt for useful
discussions on this subject.  The authors are greatly indebted to the 
referee who pointed out several errors in the presentation of the results and
helped uncover a gap in the proof of
Lemma~\ref{lemma:finiterankrespos}, which has been fixed.
Finally the authors would like to acknowledge MSRI and AIM for holding workshops on the subject of
CR complexity, which the authors attended and which led to the present
project.


\section{The basic setup} \label{section:thesetup}

We begin by recalling the relevant definitions and proving some basic
results about restrictions of real-valued polynomials.  Let
$\langle \cdot, \cdot \rangle$ denote the standard Hermitian inner product.
Let $A$ be a Hermitian symmetric matrix, then we call $\langle A z, z
\rangle$ a \emph{Hermitian form} where $z\in \C^{n+1}$.

Let 
\begin{equation}
\sZ_{n,d} \colon \bC^{n+1} \to \bC^{\binom{n+d}{d}}
\end{equation}
be the mapping whose components are all the degree $d$ monomials in $n+1$
variables $z_0,z_1,\ldots,z_n$.  That is, $\sZ_{n,d}$ is the mapping given by 
\begin{equation}
z \mapsto (z_0^d, z_0^{d-1} z_1, \ldots, z_n^d).
\end{equation}  
The mapping $\sZ_{n,d}$ is called the \emph{Veronese mapping in $n+1$ variables of degree $d$}.  

\begin{defn}
Let $L$ be an $(m+1)$-dimensional subspace of $\C^{n+1}$.  Let $E_L$ be a linear
embedding of $L$ into $\bC^{n+1}$.  Define $S^d_{E_L}$ to be the linear
mapping that makes the following diagram commute:
\begin{equation}
\begin{CD}
\bC^{m+1}     @>E_L>> \bC^{n+1}\\
@VV {\sZ_{m,d}} V       @VV {\sZ_{n,d}} V\\
\bC^{\binom{m+d}{d}}     @>S^d_{E_L}>> \bC^{\binom{n+d}{d}}.\\
\end{CD}
\end{equation}
Define $T^d_{E_L}$ to be the matrix associated to the linear mapping $S_{E_L}^d$,
then
\begin{equation}
\sZ_{n,d} \circ E_L = T^{d}_{E_L}\sZ_{m,d} .
\end{equation}  
We often abbreviate $T^d_L$ for $T^d_{E_L}$ when the particular choice
of embedding is not crucial.  
\end{defn}

Note that $T^d_{E_L}$ is not unique; it depends on the particular embedding
$E_L$.  However, if we have $E_L$ and $F_L$ two embeddings of
$L$ then we have $T^d_{E_L} = \psi \circ T^d_{F_L}$, where $\psi$ is an
isomorphism.

\begin{lemma}  \label{lemma:changecoord}
Let $r(z,\bar{z})$ be a bihomogeneous polynomial defined for $z \in
\C^{n+1}$.  If $\phi$ is a linear change of coordinates of $\bC^{n+1}$, then
\begin{equation}
\rank r \circ \phi = \rank r.
\end{equation}
\end{lemma}
\begin{proof}
We prove a more general fact in Lemma \ref{lemma:rankchangecoord}.
\end{proof}

The lemma shows that rank of $r|_{E_L}$ does not depend on the
choice of embedding $E_L$.  We record this observation in the following lemma.

\begin{lemma}  \label{lemma:welldefined}
Let $r(z,\bar{z})=\langle C \sZ_{n,d}, \sZ_{n,d} \rangle$, where $C$ is a
Hermitian matrix, and suppose $n \geq 1$.  Let $L$ be an
$(m+1)$-dimensional subspace of $\bC^{n+1}$.  Let $E_L$ and
$F_L$ be two embeddings of $L$ into $\bC^{n+1}$.  Then 
\begin{equation}
\rank T^{*}_{E_L}C T_{E_L}=\rank T^{*}_{F_L}C T_{F_L}.
\end{equation}
\end{lemma}

\begin{defn}  \label{defofTnd}
Let $H_{\vec{c}}$ denote the hyperplane in $\C^{n+1}$ given by the following defining equation: 
\begin{equation}
z_{0}=c_1 z_1+ \cdots +c_{n} z_n ,
\end{equation}
where $\vec{c} = (c_1,c_2,\ldots,c_n)$.
We embed 
$H_{\vec{c}}$ using $E_{H_{\vec{c}}}(z_1,\ldots,z_{n}) =
(\sum_{j=1}^n c_jz_{j},z_1,\ldots,z_{n},)$.

We abbreviate the notation for the restriction matrix $T_{H_{\vec{c}}}^d$ by
$T_{\vec{c}}^d$.  
\end{defn}

The restriction matrix $T_{\vec{c}}^d$ is of size $\binom{n+d}{d}\times
\binom{n+d-1}{d}$.

To study the maximum rank of restrictions of
bihomogeneous polynomials to hyperplanes, it suffices to study the
hyperplanes of the form $H_{\vec{c}}$.

\begin{lemma} \label{lemma:restrictgraphs}
Let $r(z,\bar{z})$ be a bihomogeneous polynomial defined for $z \in \C^{n+1}$
 and suppose $n \geq 1$.
Then 
\begin{equation}
\max_{H} \,\rank r|_H = \max_{\vec{c} \in \C^n} \, \rank r|_{H_{\vec{c}}} ,
\end{equation}
where $H$ ranges over all complex hyperplanes through the origin.
\end{lemma}

\begin{proof}
The set of hyperplanes of the form $H_{\vec{c}}$ is dense in the
Grassmannian $G(n,n+1)$.  The rank of the restriction
being bounded by a fixed integer is a holomorphic condition.
Therefore $\rank r|_H$ achieves the maximum on an open dense set in
$G(n,n+1)$, and the conclusion follows.
\end{proof}


\section{Green's Restriction Theorem} \label{section:greens}

The main technical tool in this paper is Green's hyperplane restriction
theorem from \cite{Green}.  We begin this section by recalling Macaulay
representations and their basic properties.  Then we give a precise
statement of Green's theorem and translate the result into the language of
this paper.  In Lemma~\ref{lemma:finiterankrespos} we will prove
a more general analytic version of Green's result using the ideas from
the proof of the restriction theorem in~\cite{Green:gin}.

Given a non-negative integer $c$ and a positive integer $d$, then $c$ is uniquely expressed
in the following form:
\begin{equation}
c= \binom{k_d}{d} +\binom{k_{d-1}}{d-1} + \cdots + \binom{k_1}{1} ,
\end{equation}
where $0 \leq k_1 < k_2 < \cdots < k_d$.  This expression is known as the $d$-th Macaulay representation
of $c$.  Following \cite{Green}, we introduce the following notation:
\begin{equation}
c_{<d>}= \binom{k_d-1}{d} +\binom{k_{d-1}-1}{d-1} + \cdots +
\binom{k_1-1}{1} ,
\end{equation}
where we define $\binom{a}{b}=0$ for $a<b$.  Notice that $\cdot_{<d>}$ is monotone increasing.



\begin{thm}[Green \cite{Green}]  \label{thm:green}
Let $E$ be a linear system in $H^0(\C\bP^{n}, \mathcal{O}(d))$ and $E_H$ be the restriction 
of $E$ to a general hyperplane $H \in G(n,n+1)$.  If $c$ is the codimension of $E$ in $H^0(\C\bP^{n}, \mathcal{O}(d))$
 and $c_H$ is the codimension of $E_H$ in $H^0(\C\bP^{n-1}, \mathcal{O}(d))$, then
\begin{equation}
c_H \leq c_{<d>}.
\end{equation}
\end{thm}

We seek a bound of the dimension of $E$ in terms of the generic dimension of the restriction.  First we restate the theorem in terms of matrices. If $E$ is a linear system in $H^0(\C\bP^{n}, \mathcal{O}(d))$ then we write:
$E = \textrm{span}\{s_1,s_2,\ldots, s_m\}$ where $s_i \in H^0(\C\bP^{n},
\mathcal{O}(d))$. Since each $s_i$ is a degree $d$ homogeneous polynomial in
$z_0,z_1,\ldots,z_n$, then we write $s_i= A_i\sZ_{n,d}$, where each $A_i$ is a vector of the right dimension. 
Then we have $\rank E = \rank A$, where $A$ is the matrix whose rows are $A_i$.

If we restrict $E$ to a hyperplane $H$ we obtain 
\begin{equation}
E|_{H} = \textrm{span}\{s_1|_{H},s_2|_{H},\ldots, s_m|_{H}\}.
\end{equation}
Then $s_i|_{H_{\vec{c}}} = A_iT_{\vec{c}}^d\sZ_{n-1,d}$, and therefore 
\begin{equation}
\rank E|_{H_{\vec{c}}} = \rank AT_{\vec{c}}^d.
\end{equation}
 If we refer to a generic choice of $H$, then equivalently we refer to a generic set of $c_i$'s, where we have:
$\rank_{\C} E|_{H_{\vec{c}}} = \rank_{\C[c_1,\ldots,c_n]} AT_{\vec{c}}^d$.  That
is, the rank of 
$E|_{H_{\vec{c}}}$ over $\C$ is equal to the rank of 
$AT_{\vec{c}}^d$ over the polynomial ring $\C[c_1,\ldots,c_n]$.

Green's theorem can now be restated as follows: Let $A$ be any matrix of size $m\times \binom{n+d}{d}$, and $N=\rank_{\bC} A$. Recall the matrix $T_{\vec{c}}^d$ is of size $\binom{n+d}{d} \times \binom{n-1+d}{d}$. Let $k = \rank_{\C[c_1,\ldots,c_n]} AT_{\vec{c}}^d$. Then we have the following relationship:
\begin{equation}
G(n,d,N) \leq k ,
\end{equation}
where 
\begin{equation}
G(n,d,N) = \binom{n+d-1}{d} - \left( \binom{n+d}{d} - N \right)_{<d>}.
\end{equation}

In the following lemma, we show that the lower bound $G$ is constant in degree.

\begin{lemma}  \label{lemma:Gconstantdegree}
If $N\leq \binom{n+d}{d}$, $d\geq 1$, and $n \geq 2$, then $G(n,d,N) = G(n,d+1,N)$.
\end{lemma}
\begin{proof}
The lemma follows immediately from
\begin{equation}
\left[ \binom{n+d+1}{d+1}-N \right]_{<d+1>} = \binom{n+d-1}{d+1} + \left[ \binom{n+d}{d}-N \right]_{<d>}.
\end{equation}
\end{proof}

By the lemma, we know that $G$ is constant in degree when it is defined. We also have $G$ is monotone in $N$ and $G(n, d, \binom{n+d}{d})=\binom{n+d-1}{d}$, hence the following bound is finite and well-defined for fixed $k$.  

\begin{defn}
Fix $n$ and $k$.
Define 
\begin{equation} \label{eq:stableR}
K_{n}(k) := \max\Bigl\{ N_0 \in \bN_0 : G(n, d, N_0) \leq k, \text{ where
$d$  is such that $N_0 \leq \binom{n+d}{d}$} \Bigr\}.
\end{equation}
\end{defn}

\begin{remark}
The function $K_n(k)$ gives the sharp dimension bound from Green's theorem.
The bound is achieved by restricting all the monomials in $n$ variables of
degree $d$.  In two dimensions, $K_2(k)$ has a particularly simple form:
\begin{equation}
K_2(k)=\frac{k(k+1)}{2}.
\end{equation}
In higher dimensions exact formulas are difficult, but we find a
simple upper bound in
Remark~\ref{remark:toprovegreen}:
\begin{equation}
K_n(k) \leq
\frac{n-1}{n} \, k^{\frac{n}{n-1}} + \frac{1}{n} \,k .
\end{equation}
Note that equality holds when $n=2$.
\end{remark}

The version of Green's theorem that we use in this paper is recorded in the following lemma.  The important point is that the bound $K_n(k)$ does not depend on the degree $d$.

\begin{lemma}  \label{lemma:green}
Let $n\geq 2$.  If $A \sZ_{n,d}$ is a linear system in $H^0(\C\bP^{n}, \mathcal{O}(d))$ with rank $N$ and $k$ is the rank of the restriction to a general hyperplane; namely, $k= \rank_{\bC[c_1,\ldots,c_n]} A T^d_{\vec{c}}$, then
\begin{equation}
N \leq K_n(k).
\end{equation}
\end{lemma}

The lemma gives an upper bound for $N$ that is
independent of $d$ for fixed $k$ and $n\geq 2$.
No such upper bound is possible if $n=1$.  


\section{Restrictions of Hermitian forms} \label{section:rhforms}

In this section, we prove Theorem \ref{finiterankres}.  To illustrate the main idea of the proof we 
prove a weaker version of the main result for hyperplanes.

\begin{lemma} \label{lemma:hyperplanesalg}
Let $n \geq 2$ and $z \in \C^{n+1}$.
There exists a function $R_n \colon \N_0 \to \N_0$, such that
if
$r(z,\bar{z})$ is a real-valued bihomogeneous, then
\begin{equation}
\rank r \leq R_{n}\left(
\max_{H} \,\rank r|_H
\right) ,
\end{equation}
where $H$ ranges over all complex hyperplanes through the origin.
\end{lemma}

\begin{proof}
Suppose $r$ is a nonzero
bihomogeneous polynomial of bidegree $(d,d)$.  We write
$r(z,\bar{z})=\langle C \sZ_{n,d}, \sZ_{n,d} \rangle$, where $C$ is the coefficient matrix.
Restricting $r$ to a hyperplane $H$, we have 
\begin{equation}
r|_H(z,\bar{z})=\langle T_H^* C T_H \sZ_{n-1,d}, \sZ_{n-1,d} \rangle.
\end{equation}
Let $k=\max_{H} \,\rank r|_H= \max_{H} \rank{T_H^*CT_H}$  where $H$ ranges over all complex hyperplanes,
then by Lemma \ref{lemma:restrictgraphs}, 
\begin{equation}
k=\max_{\vec{c}\in \bC^n} \, \rank r|_{H_{\vec{c}}}=\max_{\vec{c}} \,\rank
{(T^d_{\vec{c}})^*CT^d_{\vec{c}}} = \rank_{\bC[\bar{c}_1,\ldots, \bar{c}_n,c_1,\ldots, c_n]} {(T^d_{\vec{c}})^*CT^d_{\vec{c}}}.
\end{equation}
The strategy is to polarize (treat $c_j$ and $\bar{c}_j$ as separate
variables) and use Lemma \ref{lemma:green} twice.

Let $D= C T^d_{\vec{c}}$, $m= \rank_\bC C$, and let $k_1= \rank_{\bC[c_1,\ldots,c_n]} D$.  By Lemma \ref{lemma:green}, we have 
\begin{equation}
m \leq K_n(k_1).
\end{equation}

For a generic choice of $c_1, \ldots, c_n$, we have
\begin{equation}
k_1= \rank_\bC D = \rank_\bC D^*.
\end{equation}

Likewise, polarizing and taking generic choices of $\bar{c}_1, \ldots,
\bar{c}_n$, we have
\begin{equation}
\rank_{\bC[\bar{c}_1,\ldots, \bar{c}_n,c_1,\ldots, c_n]}
{(T^d_{\vec{c}})^*CT^d_{\vec{c}}} = \rank_{\bC[c_1,\ldots, c_n]}
{D^{*}T^d_{\vec{c}}} .
\end{equation}
Applying Lemma~\ref{lemma:green} again, we get
\begin{equation}
k_1 \leq K_n(k).
\end{equation}

Thus $\rank{r} \leq K_n(K_n(k))$, and we have an upper bound for $\rank{r}$ that is independent of degree.  We complete the proof by defining $R_n(k)$ to be the maximum $\rank{r}$ for any bihomogeneous polynomial $r$
with $\max_H \rank r|_H \leq k$.
\end{proof}

\begin{remark}
In the proof of Lemma \ref{lemma:hyperplanesalg}, we proved that 
\begin{equation}
R_n(k) \leq K_n\bigl(K_n(k)\bigr).
\end{equation}
In fact, we believe that $R_n(k) =K_n(k)$.
\end{remark}

We are now ready to prove the main result, Theorem~\ref{finiterankres}.  For
reader convenience we restate the theorem.

\begin{thmnonum}
Let $n \geq 2$, $z\in \C^{n+1}$, and $1 \leq m \leq n-1$.  Then there exists a function
$R_{m,n} \colon \N_0 \to \N_0$ such that for any bihomogeneous polynomial $r(z,\bar{z})$, 
\begin{equation}
\rank r \leq R_{m,n}\left(
\max_{L \in G(m+1,n+1)} \,\rank r|_L
\right) .
\end{equation}
\end{thmnonum}

\begin{proof}
Suppose $r$ is a nonzero
bihomogeneous polynomial of bidegree $(d,d)$.
We again write $r(z,\bar{z})=\langle C \sZ_{n,d}, \sZ_{n,d} \rangle$.
After applying a linear change of coordinates we assume without loss of
generality that all $L \in G(m+1,n+1)$ are given as
\begin{equation} \label{eq:Vdef}
z' =
V
z'' ,
\end{equation}
for $(z',z'') \in \C^{n-m} \times \C^{m+1}$
and
$V \in M_{n-m,m+1}$ is a matrix.
Using the embedding $z'' \mapsto (V z'', z'')$ we obtain a
restriction matrix $T_L$.  We notice that $T_L$ depends holomorphically on
the entries of $V$.  

We write
\begin{equation}
r|_L(z,\bar{z})=\langle T_L^* C T_L \sZ_{m,d}, \sZ_{m,d} \rangle.
\end{equation}
And again note that
\begin{equation}
\rank{r|_L}=\rank{T_L^*CT_L} .
\end{equation}
If $m = n-1$ we are done.  We let $R_{n-1,n}(k) = R_n(k)$.
Otherwise,
$T_L$
is a composition $T_L = T_S T_H$ where $H$ is a hyperplane through
the origin in $S$,
where $S$ is an $m+2$ dimensional plane through the origin in $\C^{n+1}$.
As $T_L^* C T_L = (T_H T_S)^{*} C T_S T_H$ is of rank less than $k$,
$T_S^{*} C T_S$ is of rank less than $R_{m+1}(k)$ by Lemma~\ref{lemma:hyperplanesalg}.
The theorem follows by applying Lemma~\ref{lemma:hyperplanesalg}
$(n-m)$ times.  The function $R_{m,n}(k)$ is then defined recursively
by composition.
\end{proof}

We say a set $\sL \subset G(m+1,n+1)$ is \emph{generic} if $\sL$
is not contained in any proper complex algebraic subvariety of $G(m+1,n+1)$.
That is, no nontrivial polynomial vanishes on $\sL$.
With this terminology, we have the following corollary:

\begin{cor} \label{cor:rankpospoly}
Let $n \geq 2$, $z\in \C^{n+1}$, and $1 \leq m \leq n-1$.
There exists a function
$K_{m,n} \colon \N_0 \to \N_0$ such that
for any positive semidefinite bihomogeneous polynomial $r(z,\bar{z})$, and a 
generic subset $\sL \subset G(m+1,n+1)$, we have
\begin{equation}
\rank r \leq K_{m,n}\left(
\max_{L \in \sL} \,\rank r|_L
\right) .
\end{equation}
\end{cor}

Note that $r(z,\bar{z}) \geq 0$ for all $z \in \C^{n+1}$
is not enough to conclude that
the matrix of coefficients is positive semidefinite.  On the other hand,
if the matrix is positive semidefinite, then $r(z,\bar{z}) \geq 0$
is true for all $z \in \C^{n+1}$.  See \cites{DV, D:hilbertadvances} for more on
positivity conditions.

\begin{proof}
The proof carries on as before after we make the following observation. Let $r(z,\bar{z})=\langle C \sZ_{n,d}, \sZ_{n,d} \rangle$.
If $C$ is positive semidefinite then we write $C = A^*A$.
The restriction of $r$ to $L$ is
\begin{equation}
r|_L(z,\bar{z})=\langle T_L^* C T_L \sZ_{m,d}, \sZ_{m,d} \rangle.
\end{equation}
Again note that
\begin{equation}
\rank{r|_L}=\rank{T_L^*CT_L}=\rank{AT_L} .
\end{equation}
Let $k=\max_{L \in \sL} \,\rank r|_L$.  The condition
\begin{equation} \label{eq:atlrankk}
\rank AT_L \leq k
\end{equation}
is defined by the vanishing of certain polynomials in the entries of $AT_L$.
Defining $L$ with a matrix $V \in M_{n-m,m+1}$
as in \eqref{eq:Vdef}, we notice that as $A$ is
fixed, then 
\eqref{eq:atlrankk} is defined by vanishing of certain polynomials in the
entries of $V$.  
We know that
\eqref{eq:atlrankk} is true for all $L \in \sL$, where $\sL$ is not contained
in any complex algebraic subvariety of $G(m+1,n+1)$.  The
set of corresponding matrices $V$ is 
not contained in any complex algebraic subvariety of $M_{n-m,m+1}$.  Thus
\eqref{eq:atlrankk} holds for 
all $V \in M_{n-m,m+1}$, and so for all $L \in G(m+1,n+1)$.

The conclusion follows by applying Lemma~\ref{lemma:green}
$(n-m)$ times.  The function $K_{m,n}(k)$ is then defined recursively
by composition as in the proof of Theorem~\ref{finiterankres}.
\end{proof}


\section{Real-analytic case} \label{section:racase}

We wish to decompose real-analytic functions as a difference of squared norms
of Hilbert space valued holomorphic functions.
This line of reasoning follows the ideas pioneered by
D'Angelo~\cite{DAngelo:CR}.   We introduce a rescaling that
allows us to work with bounded operators.

First we prove that the matrix of coefficients is a trace-class operator
if the
series converges in the right neighborhood.  
See the book \cite{AG:linop} for more information on matrix operators on
$\ell^2$.
We think of $\C^n$ as an open subset of $\C\bP^n$ by setting $z_0 = 1$.
Therefore in $\C^n$ our coordinates will be $z = (z_1,\ldots,z_n) \in \C^n$.
Let $\sZ = \sZ_n = (\ldots,z^\alpha,\ldots)$ be the mapping of $z \in \C^n$
to the space of infinite sequences, where components of the mapping
are all possible monomials.
By using the geometric series
we note that $\sZ$ maps the unit polydisc $\Delta$ to $\ell^2$.

\begin{lemma} \label{lemma:compactop}
Let $\Delta \subset \C^n$ be the unit polydisc.
Suppose that $r(z,\bar{z})$ is a real-analytic function whose
complexified power series at 0 converges in a neighborhood of
$\overline{\Delta} \times \overline{\Delta} \subset \C^n \times \C^n$.
Then the matrix of coefficients of $r$ defines a trace-class operator on
$\ell^2$.
\end{lemma}

\begin{proof}
Write $r$ as $\langle C \sZ , \sZ \rangle$ where $C = [ c_{\alpha \beta}
]_{\alpha\beta}$ is the matrix of coefficients.
The hypothesis says that the interior of the
domain of convergence includes the point
$z=(1,1,\ldots,1)$.  As 
the series converges absolutely
at $(1,1,\ldots,1)$ we get
$\sum_{\alpha \beta} \abs{c_{\alpha\beta}} < \infty$.  A matrix whose
entries are summable defines a compact operator.
  Moreover,
if $\{ e_\alpha \}$
is the standard basis corresponding to the monomials, then
\begin{equation}
\sum_{\alpha} \langle \abs{C} e_\alpha, e_\alpha \rangle \leq
\sum_{\alpha} \bigl\lVert \abs{C} e_\alpha \bigr\rVert =
\sum_{\alpha} \norm{C e_\alpha} \leq
\sum_{\alpha} \sum_{\beta} \abs{\langle C e_\alpha, e_\beta \rangle} < \infty .
\end{equation}
Here $\abs{C}$ is the unique positive Hermitian square root of $C^*C$.
So $C$ is trace-class.
\end{proof}

As $C$ is a Hermitian trace-class operator, we apply the spectral theorem.
We write $C$ as a sum of rank one matrices as follows.
Let $\lambda_k$ be the $k$-th nonzero eigenvalue.  Let $\epsilon_k = \pm 1$
be the sign of the $k$-th nonzero eigenvalue.  We ignore all zero eigenvalues.
Let $\{ v_k \}$ be an orthonormal set of corresponding eigenvectors.
Let $m \in \N_0 \cup \{ \infty \}$ be the rank of $C$, then we write
\begin{equation}
C = \sum_{j=1}^m \lambda_j v_j v_j^* .
\end{equation}
The sum converges since the eigenvalues of a trace-class operator are absolutely summable.
We define holomorphic functions
$f_j(z) = \sqrt{\abs{\lambda_j}} \, v_j^* \sZ$.
As $v_j \in \ell^2$, we note that each $f_j$ converges in the unit polydisc.
We see that
\begin{equation} \label{eq:holexp}
\langle C \sZ, \sZ \rangle =
\sum_{j=1}^m \epsilon_j \abs{f_j(z)}^2 .
\end{equation}
Suppose that for some $m'$, $\epsilon'_j = \pm 1$, and some $g_j$ holomorphic in
the unit polydisc we write
\begin{equation}
\langle C \sZ, \sZ \rangle =
\sum_{j=1}^{m'} \epsilon'_j \abs{g_j(z)}^2 .
\end{equation}
If $m' = \infty$, then obviously $m' \geq m$.  If $m'$ is finite
then we have written $C$ as a sum of $m'$ rank-one operators, and hence $m' \geq m$.  

Next we have to show that the rank of the matrix $C$ is invariant under
change of coordinates.  We define the rank of $r$ at the origin to be the
rank of the matrix $C$.

\begin{lemma} \label{lemma:rankchangecoord}
Let $\Delta \subset \C^n$ be the unit polydisc.
Suppose that $r(z,\bar{z})$ is a real-analytic function whose
complexified power series at 0 converges in a neighborhood of
$\overline{\Delta} \times \overline{\Delta} \subset \C^n \times \C^n$.

Suppose that $\varphi$ is a biholomorphic change of coordinates that takes
a neighborhood of $\overline{\Delta}$ to
a neighborhood of $\overline{\Delta}$.  Then
$r \circ \varphi$ (which also converges on
a neighborhood of
$\overline{\Delta} \times \overline{\Delta} \subset \C^n \times \C^n$)
has the same rank as $r$ at the origin.
\end{lemma}

\begin{proof}
Suppose that $r$ is of finite rank $k$.  Then there exist $k$
eigenvectors of $C$, and therefore we write
\begin{equation}
r(z,\bar{z}) = \sum_{j=1}^k \epsilon_j \abs{f_j(z)}^2 
\end{equation}
for $\epsilon_j = \pm 1$ and $f_j$ holomorphic functions.  Therefore
\begin{equation}
(r \circ \varphi)(z,\bar{z}) = \sum_{j=1}^k \epsilon_j \abs{f_j(\varphi(z))}^2 , 
\end{equation}
so $\rank r \circ \varphi \leq \rank r$.  We are finished by symmetry.
\end{proof}

We define the number of positive and negative eigenvalues
at a point (the rank of the positive or negative definite part) in the
obvious way.
Using exactly the same argument as above,
we can show that the rank, the number of positive, and the number of negative eigenvalues are
constant on connected sets.

\begin{lemma} \label{lemma:rankchangept}
Let $U \subset \C^n$ be a connected open set.  Let $r \colon U \to \R$
be a real-analytic function.  Let $p_1,p_2 \in U$, and let
$k_j,a_j,b_j$ denote the rank, the number of positive, and the number of
negative eigenvalues respectively at $p_j$.  Then $k_1 = k_2$, $a_1 = a_2$, and $b_1 =
b_2$.
\end{lemma}

With the aid of the lemmas above we define the rank of $r$
regardless of where it converges.  We simply translate and
dilate $r$ so that it converges in the unit polydisc and then take the rank
of the resulting compact operator.

If the domain $U$ of $r$ is connected,
we define the
\emph{rank} of $r$ as the rank at some fixed point of $U$.
Similarly define the \emph{signature pair}
to be the pair $(a,b)$ if $r$ has $a$ positive and $b$ negative eigenvalues.
We allow $a$ and $b$ to be in $\N_0 \cup \{\infty \}$.


\begin{lemma} \label{lemma:restcompact}
Let $L$ be an affine complex submanifold of $\C^n$ of dimension $m$
that intersects the unit polydisc $\Delta_n \subset \C^n$ and
$E \colon \C^{m} \to L$ an embedding such that $E$ takes the
closed unit polydisc $\overline{\Delta}_{m} \subset \C^{m}$
to a subset of $L \cap \Delta_n$.

The restriction matrix $T_L$ defined by $\sZ_n \circ E = T_L \sZ_{m}$
is a Hilbert-Schmidt operator.
\end{lemma}

\begin{proof}
Note that
$\norm{\sZ_n}^2 = \langle \sZ_n, \sZ_n \rangle$ is the absolute value squared
of the geometric series and hence converges in the unit polydisc $\Delta_n
\subset \C^n$.
As $E$ takes
the closed unit polydisc $\overline{\Delta}_{m} \subset \C^{m}$
to $\Delta_n$, we have that $\norm{\sZ_n \circ E}^2$
converges absolutely for $z \in \overline{\Delta}_m$, therefore
\begin{equation}
\norm{\sZ_n \circ E}^2 = 
\langle \sZ_n \circ E , \sZ_n \circ E \rangle =
\langle T_L^* T_L \sZ_{m}, \sZ_{m} \rangle ,
\end{equation}
where $T_L$ is the restriction matrix induced by the embedding $E$.
By Lemma~\ref{lemma:compactop}, $T_L^*T_L$ is trace-class and so $T_L$ is a Hilbert-Schmidt operator.
\end{proof}

When $r(z,\bar{z}) = \langle C \sZ , \sZ \rangle$ is positive
semi-definite, that is, has no negative eigenvalues, we write
\begin{equation*}
r(z,\bar{z}) = \langle A^*A \sZ, \sZ \rangle = \norm{A\sZ}^2 .
\end{equation*}
When everything is scaled appropriately so that $C$ is trace-class, then 
$A$ is Hilbert-Schmidt.

Next we note that we can write a finite rank holomorphic mapping
in a convenient way when working in general coordinates.
The idea is very similar to the technique used to
prove Green's theorem in \cite{Green:gin}.
Of course now we are working in nonhomogeneous
coordinates so we must make a generic choice of affine coordinates instead
of linear.  The following proposition allows us to prove our main result
by reducing to the case where each component of $A\sZ$ 
is a single monomial.

We use the same somewhat nonstandard reverse lexicographic order, or reverse lex,  as in
\cite{Green:gin}.  Reverse lex is the
common reverse lex ordering in each degree, but we order degrees by
putting the lower total degree first.   For example when $n=3$ we have
\begin{equation}
1 > z_1 > z_2 > z_3 > z_1^2 > z_1z_2 > z_2^2 > z_1z_3 > z_2 z_3 > z_3^2 >
\cdots .
\end{equation}
When writing down the monomials, for example in $\sZ$, we write then down in
decreasing order as written above.  In this ordering, the initial monomial
of a series is the maximal monomial that appears.
We call a set $\sM$ of monomials affine-Borel-fixed if whenever $z^\alpha \in \sM$
and $z_j | z^\alpha$, then $z^\alpha\frac{1}{z_j}$ and
$z^\alpha\frac{z_\ell}{z_j}$ is in $\sM$ for all $\ell < j$.
If a finite affine-Borel-fixed set of monomials 
is homogenized to the same degree via a new variable $z_0$, where $z_0 >
z_1$, then
the set is Borel-fixed in the standard
sense (that is, the affine step becomes same as the others).
As usual by pivot columns we mean the columns of the leading
terms after row reduction.
The following proposition is an analogue of Galligo's Theorem (see Theorem 1.27
in \cite{Green:gin}) to our setting.  The proof technique seems standard,
but we include it for completeness.

\begin{prop} \label{prop:generalcoords}
Let $f(z) = A \sZ$ be a finite rank holomorphic mapping, with linearly independent components, from a neighborhood
of the closed unit polydisc $\overline{\Delta} \subset \C^n$ to
$\C^N$.  That is, $A$ is an $N \times \infty$ matrix of rank $N$.  Let $\sZ$
be given in the monomial order as above (or any multiplicative monomial
order).  Then there exists
an affine self mapping $\chi$ of $\C^n$ arbitrarily close to the identity,
and an $N \times N$ invertible
matrix $B$ such that
\begin{equation}
B A \sZ \circ \chi = \widetilde{A} \sZ .
\end{equation}
where $\widetilde{A}$ is a rank $N$ matrix in reduced row echelon form such that the
pivot columns correspond to an affine-Borel-fixed set of monomials.
Furthermore, the
set of such $\chi$ is an open dense subset of a neighborhood of the
identity in the space of affine maps.
\end{prop}

\begin{proof}
Let $A_p$ denote the first $p$ columns of the matrix $A$ followed by zeros.
Let $A^\chi$ be defined by $A^\chi \sZ = A \sZ \circ \chi$,
for affine maps $\chi$.
Because rank is given by looking at certain subdeterminants,
for a fixed $p$ the rank of ${(A^\chi)}_p$ is maximal
and constant on a Zariski open set in a neighborhood of the identity in the
space of affine maps $\chi$.  
There are only finitely many $p$ for which the rank of ${(A^\chi)}_p$
is less than $N$.  
The intersection
of finitely many Zariski open sets is Zariski open.
Therefore, after applying an affine map close to the identity we can assume that
for each $p$, the rank of ${(A^{\chi})}_p$ is maximal (constant)
for all $\chi$ in a neighborhood of the identity.

Take $q$ large enough such that $A_q$ is rank $N$.  The ranks of all submatrices
of ${(A^\chi)}_q$ achieve a maximum on a Zariski open set, and therefore
we assume that the ranks of all submatrices of ${(A^\chi)}_q$ are constant 
for all $\chi$ near the identity.

The $p$ where the rank of $A_p$ increases are precisely the columns with
the pivot elements.  We will show that these columns are
affine-Borel-fixed.

Suppose for contradiction that $p$ is the first pivot column that violates the
affine-Borel-fixed property in
the following way.  Without loss of generality suppose that $p$ corresponds
to $z^\alpha$ where $z_1 | z^\alpha$.  And suppose that the $j$-th column ($j
< p$) 
corresponding to $z^\alpha \frac{1}{z_1}$ does not contain a pivot.
The other moves mentioned above are similar, the map used below need not be
affine but linear.

Write $Z = (z_2,\ldots,z_n)$, and let $\beta \in \N_0^{n-1}$ be a
multi-index for $Z$ such that $\alpha = (m+1,\beta)$ for some $m \geq 0$.
Therefore,
by the fact that $\alpha$ is the first place where the affine-Borel-fixed
property is violated, we know that there exist pivots in columns
corresponding to
$(0,\beta), (1,\beta), \ldots, (m-1,\beta)$, and in $(m+1,\beta)$,
but not at $(m,\beta)$.  The pivot at $(m+1,\beta)$ corresponds 
to the last nonzero column in $A_p$.

%
%
%
%

Without loss of generality we assume that $\rank A_p = N$.
This simplification
is possible by taking all the pivot columns in $A_p$
and in each one
picking a pivot element that we would have used in the row reduction,
and then disregarding all rows without a pivot element.
The same pivot elements can be used after the simplification and
therefore the location of pivot columns in $A_p$ is
unchanged.

Let $\gamma \in \N_0^{n-1}$ run over 
the multi-indices for $Z$.
We write $f$ as
\begin{equation}
f(z) = f(z_1,Z) = \sum_{j,\gamma} w_\gamma^j z_1^j Z^\gamma
\end{equation}
for $N$-vectors $w_\gamma^j$.
Take the affine map
that replaces $z_1$ with  $z_1 + c$.  We then have
\begin{equation}
f(z_1 + c,Z) = \sum_{j,\gamma}
\left( \sum_{\ell=0}^\infty c^\ell \binom{j+\ell}{j}
w_\gamma^{j+\ell} \right) z_1^j Z^\gamma .
\end{equation}

Let 
the multi-indices $(k_1,\gamma_1),\ldots,(k_s,\gamma_s)$, where
$k_j \in \N_0$ and $\gamma_j \in \N_0^{n-1}$, $\gamma_j \not= \beta$,
correspond to the remaining pivot columns in $A_p$.  Note $s+m+1= N$.
Because $A_{p-1}$ is of rank lower than $N = \rank A_p$,
we have 
$\det(
w_{\beta}^{0} , \ldots, w_{\beta}^{m},
w_{\gamma_1}^{k_1} , \ldots, w_{\gamma_s}^{k_s}) = 0$.
%
For simplicity, in the next calculation we write 
$W_1 = [w_{\beta}^{0}, \ldots , w_{\beta}^{m}]$ and
$W_2 = [w_{\gamma_1}^{k_1}, \ldots , w_{\gamma_s}^{k_s}]$.
We take the determinant of the same
columns after applying our affine map:
\begin{multline}
\det\Biggl(
\sum_{\ell=0}^{\infty} c^\ell
\binom{{0}+\ell}{{0}}
w_{\beta}^{0+\ell} ,
\ldots ,
\sum_{\ell=0}^{\infty} c^\ell
\binom{{m}+\ell}{{m}}
w_{\beta}^{m+\ell} ,
\\
\sum_{\ell=0}^{\infty} c^\ell
\binom{{k_1}+\ell}{{k_1}}
w_{\gamma_1}^{k_1+\ell} ,
\ldots ,
\sum_{\ell=0}^{\infty} c^\ell
\binom{{k_s}+\ell}{{k_s}}
w_{\gamma_s}^{k_s+\ell} \Biggr)
\\ =
\det(
W_1
,
W_2
)
+ 
c
\Biggl(
\binom{{0}+1}{{0}}
\det(
w_{\beta}^{{0}+1} ,
w_{\beta}^{1},
\ldots,
w_{\beta}^{m}
,
W_2
)
\\
+
\binom{{1}+1}{{1}}
\det(
w_{\beta}^{0},
 w_{\beta}^{{1}+1} ,
\ldots,
w_{\beta}^{m}
,
W_2
)
+ \cdots
+
\binom{{m}+1}{{m}}
\det(
w_{\beta}^{0},
w_{\beta}^{1} ,
\ldots,
w_{\beta}^{m+1}
,
W_2
)
\\
+
\binom{{k_1}+1}{{k_1}}
\det(
W_1
,
w_{\gamma_1}^{k_1+1},
\ldots ,
w_{\gamma_s}^{k_s}
)
+ \cdots
+
\binom{{k_s}+1}{{k_s}}
\det(
W_1
,
w_{\gamma_1}^{k_1},
\ldots ,
w_{\gamma_s}^{k_s+1}
)
\Biggr)
\\
 + \text{higher order terms in $c$}.
\end{multline}
By assumption the ranks are constant for all $c$ near zero and
so this function must be identically zero.  In particular
the coefficient of $c$ must be zero:
\begin{multline}
\det(
w_{\beta}^{{0}+1} ,
w_{\beta}^{1},
\ldots,
w_{\beta}^{m}
,
w_{\gamma_1}^{k_1},
\ldots ,
w_{\gamma_s}^{k_s}
)
+
2
\det(
w_{\beta}^{0},
 w_{\beta}^{{1}+1} ,
\ldots,
w_{\beta}^{m}
,
w_{\gamma_1}^{k_1},
\ldots ,
w_{\gamma_s}^{k_s}
)
\\
+ \cdots +
m
\det(
w_{\beta}^{0},
w_{\beta}^{1} ,
\ldots,
w_{\beta}^{(m-1)+1}, 
w_{\beta}^{m}
,
w_{\gamma_1}^{k_1},
\ldots ,
w_{\gamma_s}^{k_s}
)
\\
+
(m+1)
\det(
w_{\beta}^{0},
w_{\beta}^{1} ,
\ldots,
w_{\beta}^{m-1}, 
w_{\beta}^{m+1}
,
w_{\gamma_1}^{k_1},
\ldots ,
w_{\gamma_s}^{k_s})
\\
+
({k_1}+1)
\det(
w_{\beta}^{0} ,
\ldots,
w_{\beta}^{m}
,
w_{\gamma_1}^{k_1+1},
\ldots ,
w_{\gamma_s}^{k_s}
)
+ \cdots
+
({k_s}+1)
\det(
w_{\beta}^{0} ,
\ldots,
w_{\beta}^{m}
,
w_{\gamma_1}^{k_1},
\ldots ,
w_{\gamma_s}^{k_s+1}
)
=  0 .
\end{multline}
The first $m$ determinants are trivially zero as there are repeated columns.
We next show that for all $j=1,\ldots,s$ we get
$\det(
w_{\beta}^{0} ,
\ldots,
w_{\beta}^{m}
,
w_{\gamma_1}^{k_1},
\ldots ,
w_{\gamma_j}^{k_j+1},
\ldots ,
w_{\gamma_s}^{k_s}
)
=  0$.
Consider two cases.
First suppose that in the monomial
ordering $(k_j,\gamma_j) > (m,\beta)$.  As the columns of $A$ are
sorted in decreasing order, the column for
$(k_j,\gamma_j)$ is to the left of the column for $(m,\beta)$.
The ordering is multiplicative, and so
$(k_j+1,\gamma_j) > (m+1,\beta)$.
If $w_{\gamma_j}^{k_j+1}$ is a pivot column, then it must be equal to
$w_{\gamma_i}^{k_i}$ for some $i \not= j$, and the determinant
must be zero.  So suppose that 
$w_{\gamma_j}^{k_j+1}$ is not a pivot column.
As both $w_{\beta}^m$ and 
$w_{\gamma_j}^{k_j+1}$ are
not pivot columns, they are linear combinations of the preceeding
columns.
In particular 
$w_{\beta}^{m+1}$ is
not a vector in the 
linear combination for either
$w_{\beta}^m$ or
$w_{\gamma_j}^{k_j+1}$.
Therefore the determinant must be 0.

Let us next suppose
$(m,\beta) > (k_j,\gamma_j)$.
We write $w_{\beta}^m$ as a linear
combination of preceeding columns, and note that the
linear combination includes neither
$w_{\beta}^{m+1}$ nor
$w_{\gamma_j}^{k_j}$.  Again this implies that the determinant
must be zero.

Therefore,
\begin{equation}
\det(
w_{\beta}^{0},
w_{\beta}^{1} ,
\ldots,
w_{\beta}^{m-1}, 
w_{\beta}^{m+1}
,
w_{\gamma_1}^{k_1},
\ldots ,
w_{\gamma_s}^{k_s}
)
= 0 .
\end{equation}
This is a contradiction as $w_{\beta}^{m+1}$ is the $p$th column
that we assumed to have a pivot entry.  These are
all the pivot entries and the matrix
$[w_{\beta}^{0}, w_{\beta}^{1} , \ldots, w_{\beta}^{m-1}, w_{\beta}^{m+1},
w_{\gamma_1}^{k_1},
\ldots ,
w_{\gamma_s}^{k_s}]$
has $N$ pivots and the determinant should be
nonzero.

The other non-affine moves are similar, but instead of $z_i + c$ we take
$z_i + c z_j$ where $j < i$.
We finish by row reducing the coefficient matrix.
\end{proof}

We now prove the real-analytic version of the main theorem on rank of
Hermitian forms for positive semi-definite forms.
The proof reduces to the algebraic version.
However,
note that a key point here is that the rank of $r$ is not required to be
finite.  As before a generic subset $\sL \subset G_{m,n}$ is a subset not
contained in any proper complex algebraic subvariety of $G_{m,n}$, that is, no
nontrivial polynomial vanishes on $\sL$.

\begin{lemma} \label{lemma:finiterankrespos}
Let $n \geq 2$ and $1 \leq m \leq n-1$.
There exists a function 
$K_{m,n} \colon \N_0 \to \N_0$ with the following property.
For any $r(z,\bar{z})$ positive semi-definite real-analytic function whose
complexified power series at 0 converges in a neighborhood of
$\overline{\Delta} \times \overline{\Delta} \subset \C^n \times \C^n$,
and $\sL \subset G_{m,n}$ a generic subset so that all the corresponding
manifolds intersect $\Delta$ and
\begin{equation}
\sup_{L \in \sL} \,\rank r|_L < \infty .
\end{equation}
Then $\rank r$ is finite and moreover
\begin{equation}
\rank r \leq K_{m,n}\left(
\max_{L \in \sL} \,\rank r|_L
\right) .
\end{equation}
\end{lemma}

\begin{proof}
When referring to the homogeneous polynomial results
of sections \ref{section:thesetup} and \ref{section:greens},
let us keep in mind that we are thinking of $\C^n$ as an open subset of
$\C\bP^n$ via setting $z_0 = 1$, and hence here our variables
are $z = (z_1,\ldots,z_n) \in \C^n$.

For an affine manifold $L$ that passes through the unit polydisc, we find
an embedding $E \colon \C^{m} \to H$ such that $E$ takes the
closed unit polydisc $\overline{\Delta}_{m} \subset \C^{m}$
to a subset of $H \cap \Delta_n$.  The induced restriction matrix $T_L$
is a Hilbert-Schmidt operator by Lemma~\ref{lemma:restcompact}.

As before,
we assume that $L$ is a graph over the first $m$ variables
and we assume that it contains a point of the
set $\{ z \colon z_1 = z_2 = \cdots = z_{m} = 0, \abs{z_j} < 1 \text{ for }
j = m+1,\ldots,n \}$.  So assume that $\sL$ contains only such $L$.
As in the proof of Theorem~\ref{finiterankres} we
identify $\sL$ with a subset $\sV \subset M_{n-m,m+1}$.
For such $L$ we take a dilation of the first $m$ variables
to obtain an embedding $E$ that maps the closed unit polydisc
$\overline{\Delta}_{m}$ into
$\Delta_n$, to ensure that $T_L$ is a Hilbert-Schmidt operator.  Without
loss of generality, we also assume
that the dilation is the same for all $L \in \sL$.  Therefore
we again assume that entries of $T_L$ depend holomorphically 
on the entries of $V \in \sV$.

Let $k = \max_{L \in \sL} \,\rank r|_L$.  Write
\begin{equation}
r(z,\bar{z}) = \langle A^*A \sZ, \sZ \rangle .
\end{equation}
Thus for $L \in \sL$ we have
\begin{equation} \label{eq:atlrankalpha}
\rank A T_L \leq k .
\end{equation}
As for the finite case, this condition is defined by the vanishing
of polynomials in the entries of $A T_L$, these depend holomorphically on the
entries of $V$.  Thus \eqref{eq:atlrankalpha} is true for
all affine manifolds of dimension $m$ that intersect $\Delta$.

Unless $m=n-1$ we again write $T_L = T_S T_H$
for a hyperplane $H$ in $S$, where $S$ is
an affine manifold of dimension $m+1$.  Hence if we prove the result
of the theorem for hyperplanes, the general case follows by induction
as before.

Therefore assume that
\begin{equation} \label{eq:atHrankalpha}
\rank r|_H \leq k .
\end{equation}
for all affine hyperplanes $H$ that intersect $\Delta$.
Then we have that
$\rank A T_H \leq k$ for all such hyperplanes $H$.

We first assume that $A$ is finite rank.  The maximum rank of $r|_H$ is achieved on a generic affine hyperplane.  
After a generic affine change of variables we assume that
$H$ is given by $\{ z_n = 0 \}$.  
Let $\sZ$ be in the reverse lex monomial order.
Using Proposition~\ref{prop:generalcoords}
we assume also that $A$ is in row reduced echelon form and
the pivot columns of $A$ are affine-Borel-fixed.
If we take the identity embedding of our hyperplane, then the restriction
matrix is simply a matrix $T_H$ with zeros everywhere except a
1 in each column in the row corresponding to monomials not depending on
$z_n$.  That is, $AT_H$ contains those columns of $A$
that correspond to monomials that do not depend on $z_n$.  We compute
\begin{equation}
\rank r|_H = \rank AT_H \geq \rank \, P T_H,
\end{equation}
where $P$ is the matrix with just the pivot columns from $A$ left
and all other columns set to 0.  The matrix $P$ is a matrix for a
polynomial of degree $d$, therefore we now truncate the series,
homogenize with $z_0$ and order the monomials in $\sZ_{n,d}$ according
to reverse lex ordering.  Let us call $P'$ this new finite
matrix that only goes up to degree $d$, and $T'_H$ the new truncated finite
restriction matrix for up to degree $d$.  The pivot columns in $P'$ (the
only columns with nonzero entries) are still of course Borel-fixed.
Therefore, a generic small linear change of coordinates in $\C^{n+1}$
does not change the
pivot columns.  Furthermore, since all monomials are now homogeneous
of degree $d$, and in reverse lex ordering, once a monomial is divisible by $z_n$
all the monomials to the right are also divisible by $z_n$.  Therefore,
a small linear change of coordinates does not change the rank
of the restriction to $H = \{ z_n = 0 \}$.
We apply Lemma~\ref{lemma:green} to obtain the desired bound
\begin{equation}
\rank r = N = \rank \, P' \leq K_n(\rank \, P' T'_H) \leq K_n(\rank r|_H)
.
\end{equation}

Now let us drop the assumption that the rank of $A$ is finite.
Let $A_\nu$ be an infinite by infinite
matrix consisting of the first $\nu$ rows of $A$
followed by
all zero rows.  Obviously $A_\nu$ is of finite rank.
We have that
\begin{equation}
\rank A_\nu T_H \leq \rank A T_H \leq k .
\end{equation}
It follows that $\rank A_\nu \leq K_n(k)$ for all $\nu$.

From the proof of
Lemma~\ref{lemma:compactop} we note that the entries
of $A^*A$
are absolutely summable and hence the entries of $A$ are square summable.
Therefore $A_\nu$ converges to $A$ in the Hilbert-Schmidt norm.
Then as
the rank of $A_\nu$ is uniformly bounded by
$K_n(k)$, we obtain
\begin{equation} \label{eq:therankbound}
\rank A \leq K_n(k) .
\end{equation}
To see this fact,
simply pick a finite orthonormal set $\{ Aw_j \}$
in the range of $A$ and then notice that
$A_\nu w_j$ must converge to $A w_j$ as $\nu \to \infty$.
Thus for large enough $\nu$, the set
$\{ A_\nu w_j \}$ must be linearly independent (and hence of cardinality
less than $K_n(k)$).
\end{proof}

\begin{remark} \label{remark:toprovegreen}
We have already done quite a bit of the heavy lifting required to prove
Lemma~\ref{lemma:green} outright.
To finish the proof we must find the maximal
size of an affine-Borel-fixed set of monomials such that at most $k$ of them
do not depend on $z_n$.
Precisely this combinatorics is done in \cite{Green:gin}.
The bound $K_n(k)$ is then the maximum size of
such a set of monomials.
It is instructive to see how the bound $K_2(k) = \frac{k(k+1)}{2}$ 
when $n=2$ can be obtained.  When we restrict to $z_2 = 0$,
then in each degree $d$, only one monomial out of at most $d+1$ survives.

We use this formulation of $K_n(k)$ to obtain a simpler bound on
its size.  We look at how $K_n(k)$ increases as $k$ increases.
The largest increase is when we are allowed to take all monomials that
depend on $z_n$ up to degree $d$.  Thus,
\begin{equation}
K_n\left(\binom{n+d-1}{n-1}\right)
=
\binom{n+d}{n}
=
\left(1+ \frac{d}{n}\right)
\binom{n+d-1}{n-1} .
\end{equation}
As $K_n(k)$ grows
most when $k=\binom{n+d-1}{n-1}$, we obtain a general bound for all $k$
by looking at these specific $k$.

Let $k = \binom{n+d-1}{n-1}$.
Using the crude estimate
$k \geq {\left(\frac{n+d-1}{n-1}\right)}^{n-1}$, we get
\begin{equation}
(n-1) k^{1/(n-1)} - n + 1
\geq 
d ,
\end{equation}
and therefore
\begin{equation}
K_n(k) \leq
k \left(1+ \frac{(n-1) k^{1/(n-1)} - n + 1}{n}\right)
=
\frac{n-1}{n} \, k^{\frac{n}{n-1}} + \frac{1}{n} \,k .
\end{equation}
\end{remark}

\begin{example} \label{example:posnes}
Let us show that positivity of $r$ is necessary.  First note that
the set $\sL$ of affine complex manifolds of dimension 1
that lie in
\begin{equation}
Q(2,1) = \{ z \in \C^3 : \abs{z_1}^2 + \abs{z_2}^2 - \abs{z_3}^2 = 1 \} 
\end{equation}
forms a generic set in $G_{1,3}$.  See section \ref{section:crmaps}
for a proof of this fact.
Let
\begin{equation}
r(z,\bar{z}) = (\abs{z_1}^2 + \abs{z_2}^2 - \abs{z_3}^2)^d .
\end{equation}
Then for $z \in Q(2,1)$
\begin{equation}
r(z,\bar{z}) = (\abs{z_1}^2 + \abs{z_2}^2 - \abs{z_3}^2)^d = 1 .
\end{equation}
Therefore the rank of $r|_L$ for $L \in \sL$ is always 1.  However
the rank of $r$ is $\binom{2+d}{d}$.

For an infinite rank example, note that
$e^r$ is of infinite rank but the rank of $e^r|_L$ is 1.
\end{example}

We can now prove Theorem~\ref{thm:rankres}.  That is, we can drop positivity
of $r$ if we assume that the rank is bounded on all hyperplanes.  The proof
is very similar to the proof of Lemma~\ref{lemma:hyperplanesalg}.

\begin{proof}[Proof of Theorem~\ref{thm:rankres}]
Suppose $r$ is a nonzero Hermitian form.  We assume that $r$ converges
in a neighborhood of the closed unit polydisc as above.
Write $r(z,\bar{z})=\langle C \sZ_n, \sZ_n \rangle$.
It is again enough to prove the
theorem for hyperplanes.
Restricting $r$ to a hyperplane $H$ that intersects the closed unit polydisc,
we have 
\begin{equation}
r|_H(z,\bar{z})=\langle T_H^* C T_H \sZ_{n-1}, \sZ_{n-1} \rangle.
\end{equation}
Let $k=\max_{H} \,\rank r|_H= \max_{H} \rank{T_H^*CT_H}$  where $H$ ranges
over all complex hyperplanes.
As before, the set of hyperplanes $H_{\vec{c}}$, defined by
$1=c_1z_1 + \cdots + c_n z_n$, is open and dense in the set
of all hyperplanes.  For an open set $U \subset \C^n$ of $\vec{c}$
giving hyperplanes that intersect
the unit polydisc, we take
an embedding, which depends holomorphically on $\vec{c}$, that takes the
the closed unit polydisc into the unit polydisc.  In what follows,
$\vec{c}$ will vary over $U$.
\begin{equation}
k=\max_{\vec{c} \in U} \, \rank r|_{H_{\vec{c}}} =
\max_{\vec{c}} \,\rank
{(T^d_{\vec{c}})^*CT^d_{\vec{c}}}
=
\max_{\bar{c}_1,\ldots, \bar{c}_n,c_1,\ldots, c_n}
\,\rank
{(T^d_{\vec{c}})^*CT^d_{\vec{c}}} .
\end{equation}
Rank is defined by vanishing of certain subdeterminants.  By polarizing
we note that we can
treat $c_1,\ldots,c_n$ and $\bar{c}_1,\ldots,\bar{c}_n$
as separate variables in the second maximum above.

Let $D= C T^d_{\vec{c}}$, $m= \rank C$, and let $k_1=
\max_{c_1,\ldots,c_n} \rank D$.  Note that we still allow the ranks to be
infinite.  Applying the bound
\eqref{eq:therankbound} from the proof of
Lemma~\ref{lemma:finiterankrespos} we obtain
\begin{equation}
m \leq K_n(k_1) .
\end{equation}

For a generic choice of $c_1, \ldots, c_n$, we have
\begin{equation}
k_1 = \rank D = \rank D^*.
\end{equation}

Likewise, taking generic choices of $\bar{c}_1, \ldots,
\bar{c}_n$, we have
\begin{equation}
k=
\max_{\bar{c}_1,\ldots, \bar{c}_n,b_1,\ldots, b_n} \,
\rank {(T_{\vec{c}})^*CT_{\vec{b}}} =
\max_{b_1,\ldots,b_n} \, \rank {D^{*}T_{\vec{b}}} .
\end{equation}
Again applying the bound \eqref{eq:therankbound} we get
\begin{equation}
k_1 \leq K_n(k).
\end{equation}
As $k$ is finite, then $k_1$ is finite and hence $m\leq
K_n\bigl(K_n(k)\bigr)$.  Defining $R_n(k)$ as in the proof of Lemma \ref{lemma:hyperplanesalg} completes the proof.
\end{proof}


\section{CR mappings of hyperquadrics} \label{section:crmaps}

In this section we apply Theorem~\ref{thm:rankres} on
bounding the rank of Hermitian forms by the rank of restrictions
to prove Theorem~\ref{thm:rigidity} on rigidity of CR mappings
between hyperquadrics.

First let us note how to construct 
CR mappings of hyperquadrics from a real-analytic
function $r(z,\bar{z})$ vanishing on the hyperquadric.

\begin{prop}
Let $p \in Q(a,b)$ be a point and let
$\Omega \subset \C^{a+b}$ be a connected open set 
with $p \in \Omega$.
Let $r \colon \Omega \to \R$ be nonzero real-analytic function of finite
rank with
$A$ positive and $B$ negative
eigenvalues such that $r$ vanishes on $Q(a,b) \cap \Omega$.

Then there exists a point $q$ arbitrarily close to $p$
and a neighborhood $U \subset Q(a,b)$ of $q$,
and a real-analytic CR mapping $F \colon U \to Q(A,B-1)$
whose image is not contained in a complex hyperplane.
\end{prop}

\begin{proof}
Write
\begin{equation}
r(z,\bar{z}) =
\sum_{j=1}^{A} \abs{f_j(z)}^2
-
\sum_{j=1}^{B} \abs{g_j(z)}^2 ,
\end{equation}
with $f_j$ and $g_j$ linearly independent.
The mapping
\begin{equation}
z \mapsto \Bigl(
\frac{f_1(z)}{g_B(z)} ,
\ldots ,
\frac{f_A(z)}{g_B(z)} ,
\frac{g_1(z)}{g_B(z)} ,
\ldots ,
\frac{g_{B-1}(z)}{g_B(z)}
\Bigr)
\end{equation}
is a meromorphic mapping taking $Q(a,b)$ to $Q(A,B-1)$.  If $g_B(p) \not =
0$, we are finished.  Otherwise note that
the set $\{ g_B = 0\} \cap Q(a,b)$ is nowhere dense in $Q(a,b)$
and hence there is a point $q \in Q(a,b)$ arbitrarily close
where $g_B(q) \not= 0$ and a holomorphic mapping defined near $q$ taking
$Q(a,b)$ to $Q(A,B-1)$.
\end{proof}

Therefore if $r$ vanishes on $Q(a,b)$ and has signature pair $(A,B)$
then $r$ induces a mapping from $Q(a,b)$ to $Q(A,B-1)$.

Let us prove a lemma about signature pairs of functions vanishing on
$Q(a,b)$ before proving 
Theorem~\ref{thm:rigidity}.

\begin{lemma} \label{lemma:rigidsigpair}
Let $a \geq 2$,  $b \geq 1$, and $a > b$.
Let $\Omega \subset \C^{a+b}$ be a connected open
set such that $\Omega \cap Q(a,b)$ is nonempty and
$r \colon \Omega \to \R$ be a real-analytic function
vanishing on $\Omega \cap Q(a,b)$ with signature pair
$(A,B)$.
If $B < \infty$ then
\begin{equation}
A \leq M(a,b,B) ,
\end{equation}
where $M = M(a,b,B)$ is a constant depending only on $a$, $b$, and $B$.
\end{lemma}


\begin{proof}
Write
\begin{equation}
r(z,\bar{z}) =
\sum_{j=1}^{A} \abs{f_j(z)}^2
-
\sum_{j=1}^{B} \abs{g_j(z)}^2 
=
\norm{f(z)}^2-\norm{g(z)}^2 ,
\end{equation}
such that the components of $f$ and $g$ are linearly independent.
Let us suppose that $B$ is finite, but for now we allow $f$ to map into
an infinite dimensional Hilbert space.

After an affine change of
coordinates we assume that we are working in a neighborhood of the
origin and that $f$ and $g$ converge in a neighborhood of a closed unit
polydisc $\overline{\Delta}$.  Let $\tilde{Q}(a,b)$
 and $\tilde{\Omega}$
correspond to $Q(a,b)$ and $\Omega$
after this affine change of coordinates.

Let $\sZ$ be the infinite vector of all monomials in $z$.  We find
an
$A \times \infty$ matrix $F$ such that $f(z) = F \sZ$.  Similarly 
we have a $B \times \infty$ matrix $G$ such that $g(z) = G \sZ$.
We know that $\norm{f(z)}^2$ converges in the closed unit polydisc, and so
by Lemma~\ref{lemma:compactop} $F$ defines a Hilbert-Schmidt
operator ($F^*F$ is trace-class) on $\ell^2$.
Similarly $G$ defines a Hilbert-Schmidt operator.

For $z \in \tilde{Q}(a,b) \cap \tilde{\Omega}$
we have 
\begin{equation}
\norm{F \sZ}^2 - \norm{G\sZ}^2 =
\norm{f(z)}^2-\norm{g(z)}^2 = 0 .
\end{equation}
As $f$ and $g$ are linearly independent we have
\begin{equation}
\rank F = A  \qquad \text{and} \qquad \rank G = B .
\end{equation}

Let $L$ be an affine complex submanifold of $\tilde{Q}(a,b)$.
Taking the right embedding of $L$ and applying Lemma~\ref{lemma:restcompact},
we find that the restriction matrix $T_L$ is a Hilbert-Schmidt operator.
Restricted to $L$ we have that
\begin{equation}
\norm{F \sZ}^2 = \norm{G\sZ}^2 .
\end{equation}
We thus have that $T_L F^*F T_L = T_L G^* G T_L$.  In other words,
the rank of $G T_L$ is equal to the rank of $F T_L$.  The rank of $G T_L$
is at most $B$.  That means that the rank of $F T_L$ is
at most $B$ for all $L$ that lie in $\tilde{Q}(a,b)$.

Let $(z',z'') \in \C^a \times \C^b$ be coordinates such that
$Q(a,b)$ is defined by $\norm{z'}^2 - \norm{z''}^2 = 1$.
We assume that $a > b$.
The defining function for $L$ that lie in $Q(a,b)$ is
\begin{equation}
z' =
V
\begin{bmatrix}
z'' \\ 1
\end{bmatrix},
\end{equation}
where $V \in M_{a,b+1}$ (an $a \times (b+1)$ matrix) has orthonormal columns.
Let us denote by $L_V$ the affine manifold defined by such a $V$.

We claim that
the set $\sV \subset M_{a,b+1} \cong \C^{a(b+1)}$
of matrices with orthonormal columns 
is not contained in any proper complex algebraic subvariety of $\C^{a(b+1)}$.
Suppose that there was a polynomial $p$ in $a(b+1)$ variables
that vanishes on $\sV$.  If we multiply the $k$-th column by $e^{i\theta_k}$
we note that $p$ still has to vanish.  By the uniqueness theorem for
Fourier series, there has to be a polynomial that is
independently homogeneous in entries of the $k$-th column (for all columns).
That is, we can multiply each column independently by any complex number
and still stay in the zero locus of $p$.  The claim follows by working on
each column independently.

The claim shows that the set $\sL$ of affine manifolds $L$ corresponding
to $V \in \sV$ is a generic set in $G_{m,n}$.
If we let $\tilde{\sL} \subset G_{m,n}$ correspond to the
set of affine manifolds 
that lie in $\tilde{Q}(a,b)$, we see that $\tilde{\sL}$ must be generic
as well.

We have all the ingredients to apply Lemma~\ref{lemma:finiterankrespos} to
obtain that
\begin{equation}
A = \rank F \leq K_{b,a+b}(B).
\end{equation}
\end{proof}

Let us restate 
Theorem~\ref{thm:rigidity}
for reader convenience before proving it.

\begin{thmnonum}
Let $a > b \geq 1$, $U \subset Q(a,b)$ be a connected open set, and 
$f \colon U \to Q(A,B)$ be a real-analytic CR mapping such that
$f(U)$ does not lie in a complex hyperplane then
\begin{equation}
A \leq N(a,b,B) ,
\end{equation}
where $N = N(a,b,B)$ is a constant depending only on $a$, $b$, and $B$.
\end{thmnonum}

\begin{proof}
Let $\varphi = (f,g)$ denote a real-analytic CR mapping of $U \subset Q(a,b)$ to
$Q(A,B)$, where
$f \colon U \to \C^A$ and
$g \colon U \to \C^B$.
Plugging into the defining equation of the target hyperquadric we obtain
that for $z \in U$
\begin{equation}
\norm{f(z)}^2-\norm{g(z)}^2-1 = 0 .
\end{equation}

Real-analytic CR mappings extend to holomorphic
mappings on a neighborhood of $U$.  Therefore we assume that $f$ and
$g$ are holomorphic in a neighborhood of $U$.
We let $r(z,\bar{z}) = \norm{f(z)}^2-\norm{g(z)}^2-1$.

The condition that $\varphi(U)$
does not lie in a complex hyperplane is simply stating that
components of the mapping $(f,g,1)$ are linearly independent.
In particular we see that $r$ has rank $A+B+1$, and signature
pair $(A,B+1)$.

We apply Lemma~\ref{lemma:rigidsigpair} to $r$ to obtain
\begin{equation}
A \leq K_{b,a+b}(B+1).
\end{equation}
\end{proof}


\section{Spheres and hyperquadrics in Hilbert space} \label{section:hilbert}

Forstneri\v{c}~\cite{Forstneric:86} has shown that there exist strictly
pseudoconvex real-analytic compact hypersurfaces
that cannot be embedded into a sphere of
any finite dimension.  On the other hand, Lempert~\cite{Lempert} has shown
that a mapping exists if one takes the sphere in the infinite
dimensional Hilbert space $\ell^2$.  Let us make the following
natural definitions to extend hyperquadrics to $\ell^2$:
\begin{align}
& Q(\infty,b) := \Bigl\{ z \in \ell^2 : - \sum_{j=1}^b \abs{z_j}^2 +
\sum_{j=b+1}^\infty \abs{z_j}^2 = 1 \Bigr\} ,
\\
& Q(\infty,\infty) := \Bigl\{ z \in \ell^2 : 
\sum_{j=1}^\infty (\abs{z_{2j-1}}^2 - \abs{z_{2j}}^2 ) = 1 \Bigr\} .
\end{align}
We write $S^\infty = Q(\infty,0)$ for the unit sphere in $\ell^2$.

By a real-analytic CR mapping from a CR manifold $M$ to $Q(\infty,b)$ we 
mean a holomorphic mapping to $\ell^2$ defined on a neighborhood of $M$
taking $M$ to $Q(\infty,b)$.

We use the ideas from the previous section to construct mappings.
For example, consider the Hermitian form
\begin{equation}
r(z,\bar{z}) = \frac{e^{\abs{z_1}^2+\abs{z_2}^2} - 1}{e-1}-1.
\end{equation}
The function $r$ vanishes precisely on the unit sphere in $\C^2$.
In fact $\frac{e^{\abs{z_1}^2+\abs{z_2}^2} - 1}{e-1}$ is 1 on the unit
sphere and is positive semi-definite.  Hence
$r$ has signature pair $(\infty,1)$; that is, $r$ has 
infinitely positive eigenvalues
and exactly one negative eigenvalue corresponding to
the $-1$.  The form induces a real-analytic CR mapping
\begin{equation}
f \colon S^2 \to S^\infty .
\end{equation}
We expand the series:
\begin{equation}
\begin{split}
\frac{e^{\abs{z_1}^2+\abs{z_2}^2} - 1}{e-1}
 & =
\frac{(\abs{z_1}^2+\abs{z_2}^2)}{e-1}
+
\frac{(\abs{z_1}^2+\abs{z_2}^2)^2}{2(e-1)}
+
\frac{(\abs{z_1}^2+\abs{z_2}^2)^3}{6(e-1)}
+ \cdots
\\
& =
\abs{\frac{z_1}{\sqrt{e-1}}}^2+\abs{\frac{z_2}{\sqrt{e-1}}}^2 +
\\
& \hspace*{1cm} +
\abs{\frac{z_1^2}{\sqrt{2(e-1)}}}^2 +
\abs{\frac{z_1z_2}{\sqrt{2(e-1)}}}^2 +
\abs{\frac{z_2^2}{\sqrt{2(e-1)}}}^2 +
\cdots
\end{split}
\end{equation}
Therefore
\begin{multline}
f(z_1,z_2) = 
\Bigl(
\frac{z_1}{\sqrt{e-1}},
\frac{z_2}{\sqrt{e-1}},
\frac{z_1^2}{\sqrt{2(e-1)}},
\frac{z_1z_2}{\sqrt{2(e-1)}},
\frac{z_2^2}{\sqrt{2(e-1)}},
\\
\frac{z_1^3}{\sqrt{6(e-1)}},
\frac{z_1^2z_2}{\sqrt{6(e-1)}},
\frac{z_1z_2^2}{\sqrt{6(e-1)}},
\frac{z_2^3}{\sqrt{6(e-1)}}, \ldots \Bigr) .
\end{multline}
The image is not contained in a hyperplane; this fact can
be seen directly by uniqueness of power series.

By taking
\begin{equation}
r_2(z,\bar{z}) = r(z,\bar{z}) + \abs{\varphi(z)}^2 r(z,\bar{z})
\end{equation}
for an arbitrary holomorphic function $\varphi(z)$, we obtain a new form
with signature pair $(\infty,2)$, and therefore a CR mapping
$f \colon S^{3} \to Q(\infty,1)$,
whose image is not contained in a hyperplane.  By repeating the procedure
we construct mappings
\begin{equation}
f \colon S^{3} \to Q(\infty,b)
\end{equation}
for any finite $b$.
Finally, by considering the form $r =
\sin\bigl(\frac{\pi}{2}(\abs{z_1}^2+\abs{z_2}^2)\bigr)$ we obtain
a mapping from $S^3$ to $Q(\infty,\infty)$.

As in the finite dimensional case, the situation changes when the
source is a hyperquadric not equivalent to a sphere.
If we apply Lemma~\ref{lemma:rigidsigpair}, we notice the following immediate
consequence.

\begin{cor}
Let $a \geq 2$, $b \geq 1$.  Let $U \subset Q(a,b)$ be a connected open
set and $f \colon U \to Q(\infty,B)$, where $B \in \N_0 \cup \{ \infty \}$,
be a real-analytic CR mapping
such that $f(U)$ is not contained in any complex hyperplane of $\ell^2$.
Then $B = \infty$.
\end{cor}

Therefore, there is no mapping from $Q(a,b)$ for finite $a$ and $b$ into
a $Q(\infty,B)$ for finite $B$.  This rigidity result does not
generalize to other hypersurfaces with indefinite nondegenerate
Levi-form.  In this case we can construct mappings to
$Q(\infty,B)$ for finite $B$.

For example, define $r$ as
\begin{equation}
r(z,\bar{z}) =
e^{\abs{z_1+1}^2+\abs{z_2}^2} - e - \abs{z_3}^2 .
\end{equation}
By expanding the series we see that
the signature pair of $r$ is $(\infty,2)$.  Furthermore, we look
at the series up to the quadratic terms
\begin{equation}
r(z,\bar{z}) = 
2e \Re z_1 + \abs{z_2}^2-\abs{z_3}^2 + \text{higher order terms} .
\end{equation}
The Levi form of $M = \{ r = 0 \}$ at the
origin is indefinite;
that is, it has one positive and one negative eigenvalue.  As the signature
pair is $(\infty,2)$ we obtain a real-analytic CR mapping
\begin{equation}
f \colon M \to Q(\infty,1) ,
\end{equation}
whose image is not contained in a complex hyperplane.
Similarly we can easily obtain a submanifold with $b$ negative eigenvalues
and arbitrary number of positive eigenvalues
in the Levi-form that admits a real-analytic CR mapping to $Q(\infty,b)$
whose image is not contained in a hyperplane.

The same construction can be used in the finite dimensional case.
By truncating the series expansion of $r$, we obtain a polynomial
$r'$ and a manifold $M' = \{ r'=0 \}$ with an indefinite Levi-form at $0$, and
such that $r'$ has signature pair $(A,2)$ for an arbitrarily high finite $A$.
We obtain a mapping from $M'$ to $Q(A,1)$ whose image is not
contained in a hyperplane.


\section{Construction of mappings} \label{section:construct}

Let us give some explicit examples in nonhomogeneous coordinates
before we give a general method for constructing mappings.

\begin{example}
Let $(z,w) \in \C^{a} \times \C^b$.
We tensor with the identity (which we write as $(z,w)$)
on the $z_a$ variable.  That is, we construct
the mapping
\begin{equation}
(z,w) \mapsto
\bigl(z_1, \ldots, z_{a-1}, z_a \otimes (z, w), w \bigr) .
\end{equation}
We 
obtain a mapping taking
$Q(a,b)$ to $Q(2a-1,2b)$.
That is,
\begin{equation}
(z,w) \mapsto
\bigl(z_1,\ldots,z_{a-1},
z_a z_1, \ldots, z_a^2,
z_a w_1, \ldots, z_a w_b,
w_1, \ldots, w_b \bigr) .
\end{equation}
Therefore the super-rigidity result of Baouendi, Ebenfelt, and Huang
\cite{BEH:hq} is sharp.  In the same language as in the introduction
we have that $A=2a-1$ and $B=2b$.
On the other hand, by Theorem~\ref{thm:rigidity} we are
not able to make $A$ arbitrarily large while keeping $B=2b$.
\end{example}

\begin{example}
We could repeat the procedure of the first example any number of times.
We 
take a mapping $\varphi \colon Q(a,b) \to Q(A,B)$ and construct
\begin{equation}
\bigl(\varphi_1, \ldots, \varphi_{A} \otimes (z,w), 
\varphi_{A+1}, \ldots, \varphi_{A+B} \bigr) .
\end{equation}
We obtain 
\begin{multline}
(z,w) \mapsto
\bigl(\varphi_1(z,w), \ldots, \varphi_{A-1}(z,w),
z_1 \varphi_{A}(z,w), \ldots
z_a \varphi_{A}(z,w),
\\
w_1 \varphi_{A}(z,w), \ldots
w_b \varphi_{A}(z,w),
\varphi_{A+1}(z,w), \ldots,
\varphi_{A+B}(z,w) \bigr)
\end{multline}
taking
$Q(a,b)$ to $Q(A+a-1,B+b)$.
\end{example}

For the constructions,
it is easier to work in homogeneous coordinates.
We define the homogeneous hyperquadric
\begin{equation}
HQ(a,b) := \Bigl\{ z \in \C^{a+b} : \sum_{j=1}^a \abs{z_j}^2 - 
\sum_{j=a+1}^{a+b} \abs{z_j}^2 = 0 \Bigr\} .
\end{equation}
The set $HQ(a,b)$ defines a real hypersurface in the projective
space $\C\bP^{a+b-1}$.
By setting $z_{a+b} = 1$ we obtain $Q(a,b-1)$ in $\C^{a+b-1}$.
The homogeneous setting is more symmetric, although one has to be careful when
converting back to the nonhomogeneous situation.

Suppose
that $r(z,\bar{z})$ is a bihomogeneous polynomial and vanishes on $HQ(a,b)$.
If $r$ has signature pair $(A,B)$ (that is $A$ positive and
$B$ negative eigenvalues), then we find
homogeneous polynomials $f \colon \C^{a+b} \to \C^A$ and $g \colon \C^{a+b}
\to \C^B$ with linearly independent components such that
\begin{equation}
r(z,\bar{z}) = {\lVert f(z) \rVert}^2 - {\lVert g(z) \rVert}^2;
\end{equation}
therefore the mapping $z \mapsto \bigl(f(z),g(z)\bigr)$ takes $HQ(a,b)$
to $HQ(A,B)$.  By dividing by $g_B$ and setting $z_{a+b} = 1$
we obtain a rational CR mapping $F \colon Q(a,b-1) \to Q(A,B-1)$.  If
the components of $(f,g)$ are linearly independent then
the image of $F$ does not lie in a complex hyperplane.

Therefore when constructing examples we will consider homogeneous polynomial mappings from $HQ(a,b)$ to
$HQ(A,B)$ with linearly independent components.


\section{Stability} \label{section:stability}

We wish to show that as long as the number of positive and negative
eigenvalues of the target hyperquadric grow in a comparable way, then
nontrivial hyperquadric mappings always exist.  We have the following
stability result.
We make the normalization $a \geq b$ for $HQ(a,b)$.

\begin{thm} \label{thm:stability}
Fix integers $a \geq b \geq 2$.  There exists an $M$ ($M=a^2+ab-2a+1$ 
suffices) such that 
when $A+B \geq M$, where $A, B \geq 2$ and
\begin{equation}
\frac{B-b+1}{A} \geq \frac{b-1}{a}
\qquad
\text{and}
\qquad
\frac{A-b+1}{B} \geq \frac{b-1}{a},
\end{equation}
then there exists a homogeneous polynomial mapping $F \colon HQ(a,b) \to HQ(A,B)$
with linearly independent components.  That is, there exists
a rational CR mapping $\tilde{F} \colon Q(a,b-1) \to Q(A,B-1)$
whose 
image is not contained in a complex hyperplane.
\end{thm}

In \cite{DL:hermsym} it is proved that when the source is a sphere
(that is, $HQ(a,1)$) then such a rational
mapping always exists for large $A+B$ provided
only that $A,B \geq 1$.  In the hyperquadric case, this sort of
stability result holds in a sector, where the size of the sector depends
on $a$ and $b$.  There may exist mappings outside of the sector given above;
however,
by Theorem~\ref{thm:rankres} it is clear that any sector in
which existence holds must make an acute angle.

The signature pairs of mappings that are constructed in the proof along
with the inequalities are illustrated in
Figure~\ref{fig:stability}.

\begin{figure}[h!t]
\begin{center}
\includegraphics{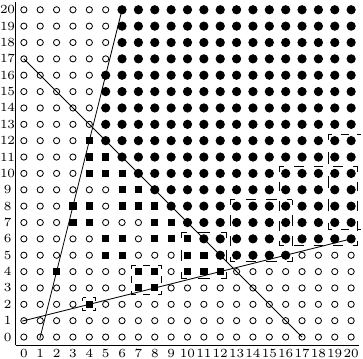}
\caption{Diagram of mappings constructed in the proof of
Theorem~\ref{thm:stability} for the source hyperquadric $HQ(4,2)$.
The lattice represents the positive quadrant of the $(A,B)$
plane.  A black dot represents that a mapping to $HQ(A,B)$ exists
by applying the theorem.  A black square represents that
a mapping was also constructed in the proof.  A circle represents a position
in the lattice for which the theorem tells us nothing.
The three lines represent the limits of the stability region from the theorem.
The dashed squares represent one succession of the ``squares'' constructed in
the proof.
\label{fig:stability}}
\end{center}
\end{figure}

\begin{proof}
Suppose $p(x)$
is a homogeneous polynomial
in the variables $x_1,\ldots,x_{a+b}$
with $A$ positive and $B$ negative coefficients such that $p(x) = 0$
when $x_1 + \cdots + x_a - x_{a+1} - \cdots - x_{a+b} = 0$.  We call
such polynomials \emph{admissible} for $(A,B)$.  
We define 
\begin{equation}
s := x_1 + \cdots + x_a - x_{a+1} - \cdots - x_{a+b} .
\end{equation}
If $p$ is admissible for $(A,B)$ then
we let $x_k = \abs{z_k}^2$ in $p$.  We obtain a Hermitian symmetric
polynomial
\begin{equation}
r(z,\bar{z}) = p(x) .
\end{equation}
As distinct monomials are linearly independent,
the rank of $r$ is equal to the number of distinct monomials of $p$, that is
$A+B$ and the signature pair is $(A,B)$.  As $r$ vanishes
on $s=0$, we have that $r$ induces a CR mapping of
$HQ(a,b)$ to $HQ(A,B)$ such that the image is not contained in a complex
hyperplane.

Therefore we work with admissible
polynomials $p$.  We fix $a$ and $b$
and we work in the
plane $\N_0^2$
and say $(A,B) \in E \subset \N_0^2$ whenever
an admissible $p$ exists for $(A,B)$ (that is, a mapping to $HQ(A,B)$ exists).

We note that $(a,b) \in E$
and $(b,a) \in E$.  If $p$ is admissible for $(A,B)$ of degree $m$
then $\tilde{p} = x_1^{k-m}p + x_2^{k-1} s$ vanishes when $s$ vanishes.
If we pick $k$ large enough then the monomials of 
$x_1^{k-m}p$ and $x_2^{k-1} s$ are distinct and
$\tilde{p}$ is admissible for $(A+a,B+b)$.  By using $-s$ instead of $s$
we build an admissible polynomial for $(A+b,B+a)$.  Therefore
for any $n,k \in \N_0$
\begin{equation}
(a,b)+n(a,b)+k(b,a) \in E
\qquad \text{and} \qquad
(b,a)+n(a,b)+k(b,a) \in E.
\end{equation}

We get a lattice of admissible points.  We claim
that for each
point $(b,a)+n(a,b)+k(b,a)$ we also obtain that the ``square''
\begin{equation}
\{ (b,a)+n(a,b)+k(b,a)-(j,m) : 0 \leq j, m \leq n+k \}
\end{equation}
is a subset of $E$.
Once we have the claim, it is a matter of noticing that
the squares must overlap for large enough $n+k$.  The inequalities from
the theorem follow.  See Figure~\ref{fig:stability} for illustration of the
proof.

To prove the claim,
we need to show that if
$(A,B)$ is admissible, then the points
\begin{equation}
\begin{aligned}
& (A,B) + (a,b), & &
(A,B) + (a,b-1), \\
& (A,B) + (a-1,b), & &
(A,B) + (a-1,b-1), \\
\end{aligned}
\end{equation}
and
\begin{equation}
\begin{aligned}
& (A,B) + (b,a),  & &
(A,B) + (b,a-1), \\
& (A,B) + (b-1,a),  & &
(A,B) + (b-1,a-1) .
\end{aligned}
\end{equation}
are also admissible.

We have already seen that
the first point belongs to $E$.  Take
$p$ that is admissible for $(A,B)$.  By dividing through
by $x_{a+b}$ we can suppose that 
$p$ contains a monomial $m$ that does not depend on $x_{a+b}$.
Furthermore we assume that $m$ has a positive coefficient,
since
if no such monomial existed, setting $x_{a+b} = 0$ would
obtain a contradiction.

Write $p = m + q$ and construct
\begin{equation}
\tilde{p} = m(s+x_{a+b}) + x_{a+b} q .
\end{equation}
Note that $s+x_{a+b} = x_1 + \cdots - x_{a+b-1}$.
The polynomial $\tilde{p}$ vanishes when $s$ vanishes.
As the coefficient of $m$ is positive then
$\tilde{p}$ is admissible for $(A,B)+(a-1,b-1)$.

Similarly we construct
\begin{equation}
\hat{p} = \frac{m(s+x_{a+b})}{2} + \frac{x_{a+b} m}{2} + x_{a+b} q .
\end{equation}
The polynomial $\hat{p}$ vanishes when $s$ vanishes.
As the coefficient of $m$ is positive then
$\tilde{p}$ is admissible for $(A,B)+(a,b-1)$.

All the other points in the claim follow by variations
on the above construction.
\end{proof}


\def\MR#1{\relax\ifhmode\unskip\spacefactor3000 \space\fi%
  \href{http://www.ams.org/mathscinet-getitem?mr=#1}{MR#1}}

\end{document}